\documentclass[11pt]{amsart}

\usepackage{amssymb}
\usepackage{verbatim}
\usepackage{hyperref}

\numberwithin{equation}{section}
\newtheorem{theorem}{Theorem}[section]
\newtheorem{lemma}[theorem]{Lemma}
\newtheorem{proposition}[theorem]{Proposition}
\newtheorem{corollary}[theorem]{Corollary}

\theoremstyle{definition}

\newtheorem{remark}[theorem]{Remark}
\newtheorem{example}[theorem]{Example}

\newtheorem{problem}[theorem]{Problem}

\newcommand{\pr}{\mathbb{P}}
\renewcommand{\P}{\mathbb{P}}
\newcommand{\QQ}{\mathbb{Q}}
\newcommand{\calo}{{\mathcal O}}

\newcommand{\cali}{{\mathcal I}}

\newcommand{\enrs}{e_{n,r,s}}

\newcommand{\conjA}{A}
\newcommand{\conjB}{B}
\newcommand{\conjC}{C}

\def\PP{\mathbb{P}}
\def\sys{\mathcal{L}}

\newtheorem{thevarthm}[theorem]{\varthmname}

\newenvironment{varthm*}[1]{\trivlist\item[]{\bf #1.}\it}{\endtrivlist}
\newenvironment{proofof}[1]{\trivlist\item[\hskip\labelsep{\it Proof of #1.}]}{\hspace*{\fill}$\Box$\endtrivlist}

\begin{document}


\title[Linear subspaces and symbolic powers]{Linear subspaces, symbolic powers and Nagata type conjectures}

\author{Marcin Dumnicki}
\address{Jagiellonian University\\
Institute of Mathematics\\
{\L}ojasiewicza 6\\
30-059 Krak\'ow}
\email{Marcin.Dumnicki@im.uj.edu.pl}

\author{Brian Harbourne}
\address{Department of Mathematics \\
University of Nebraska--Lincoln\\
Lincoln, NE 68588-0130, USA }
\email{bharbour@math.unl.edu}

\author{Tomasz Szemberg}
\address{Instytut Matematyki UP\\
PL-30-084 Krak\'ow, Poland}
\email{szemberg@ap.krakow.pl}

\author{Halszka Tutaj-Gasi\'nska}
\address{Jagiellonian University\\
Institute of Mathematics\\
{\L}ojasiewicza 6\\
PL-30348 Krak\'ow}
\email{htutaj@im.uj.edu.pl }

\date{September 28, 2012}

\begin{abstract}
Inspired by results of Guardo, Van Tuyl and the second author for lines in $\pr^3$,
we develop asymptotic upper bounds for the least degree of a homogeneous
form vanishing to order at least $m$ on a union of disjoint
$r$ dimensional planes in $\pr^n$ for $n\geq 2r+1$.
These considerations lead to new conjectures that suggest that the well
known conjecture of Nagata for points in $\P^2$ is not
an exotic statement but rather a manifestation of
a much more general phenomenon which seems to have been overlooked so far.
\end{abstract}

\maketitle



\section{Introduction}
   It has been anticipated that the statement of the Nagata Conjecture (see Remark \ref{Nagrem})
   generalizes to points in higher dimensional projective space, e.g.
   in \cite{refI}, \cite{LafUga07}. In the present paper we show
   that the pioneering idea of Guardo, Van Tuyl and the second author
   to replace points by lines
   \cite{refGHVT}, extends to the general setting of linear subspaces
   of arbitrary dimension. Corollary \ref{cor:diff eq tower} exhibits
   in addition surprising ties between emerging structures.
   A posteriori our research provides interesting information
   on the shape of the nef cone of projective spaces blown up
   along a collection of disjoint linear subspaces (see Remark \ref{rmk:nef cone}).

For any nonzero homogeneous ideal $J\subset k[\pr^n]$,
let $J_t$ be the $k$-vector space span of the homogeneous forms in $J$ of degree $t$,
and let $\alpha(J)$ be the least $t$ such that $J_t\neq 0$, so $\alpha(J)$
is the degree of a nonzero element of least degree in $J$.
These invariants have been studied for a long time, see e.g. \cite{EV83}, \cite{refW}.
An asymptotic counterpart which we refer to as the \emph{Waldschmidt constant},
was studied by Waldschmidt \cite{refW} in case $J$ is an ideal of
a finite number of points. We define it more generally as
\begin{equation}\label{eq:WalConst}
\gamma(J):=\lim_{m\to\infty}\frac{\alpha(J^{(m)})}{m} =\inf_{m}\frac{\alpha(J^{(m)})}{m},
\end{equation}
where $J^{(m)}$ denotes the $m$th symbolic power of $J$.
See \cite[Lemma 2.3.1]{refBH} for the existence of the limit in \eqref{eq:WalConst} for arbitrary nonzero
homogeneous ideals, as well as for a justification of the second equality in \eqref{eq:WalConst}.

It is usually very hard to compute $\alpha$ for large symbolic powers of an ideal,
with the result that there are not many ideals for which the value of $\gamma$ is known.
This makes it of interest to compute values of $\gamma$ in specific examples, and even to
obtain bounds on $\gamma$. In the present paper, motivated by this interest and by
the results of \cite{refGHVT} for lines in $\pr^3$, we study Waldschmidt constants for
ideals $J_{n,r,s}$ of unions of $s$ generic $r$--dimensional
subspaces in the projective space $\P^n$, where by \emph{linear subspace} we
mean a subvariety defined by $n$ or fewer linear forms.
We introduce polynomials $\Lambda_{n,r,s}(x)\in \QQ[x]$ (see Proposition \ref{lem:Lambdan0s})
tied together by a tower of differential equations (see Corollary \ref{cor:diff eq tower})
such that
$$\gamma(J_{n,r,s})\leq g_{n,r,s},$$
where $g_{r,n,s}$ is the largest real root of $\Lambda_{n,r,s}$, and we conjecture that
equality holds for $s\gg0$.
In particular, this conjecture implies that $\gamma(J_{n,r,s})$ is algebraic over the rationals for $s\gg0$.
This conjecture generalizes and extends famous conjectures of Nagata and Iarrobino; see Remarks \ref{Nagrem} and \ref{rem:conjs}.

The content of this paper is as follows. In Section 2 we prove
our main theorem, Theorem \ref{thm:zeroes of Lambda},
giving upper bounds on $\gamma(I)$ for ideals of disjoint unions of $r$-planes in $\pr^n$, and we state
three conjectures, including the one referred to above.
In Section 3 we prove one of our conjectures, Conjecture {\conjB}, in the case that $r\leq 1$.
In Section 4 we work out an example giving support for Conjecture {\conjC}.
We also include two appendices. The first contains results involving binomial
coefficients needed for the proof of Theorem \ref{thm:zeroes of Lambda}.
The second appendix includes computations
demonstrating various techniques for computing $\gamma(J_{n,r,s})$
for various specific values of $n,r,s$. Explicit
values of $\gamma(J_{n,r,s})$ are of interest even when $s$ is not necessarily large.

\section{Hilbert functions of disjoint fat flats}
Let $k$ be an algebraically closed field,
$L_1,\ldots,L_s\subset \pr^n_k$ distinct linear subspaces of dimension $r$, and let
$I=I(L_1\cup\cdots\cup L_s)=I(L_1)\cap\cdots\cap I(L_s)\subset k[\pr^n]=k[x_0,\ldots,x_n]$
be the ideal generated by all forms that vanish on each $L_i$. In this situation the $m$th symbolic
power of $I$ has a simple form:
$$I^{(m)}=I(L_1)^m\cap\cdots\cap I(L_s)^m.$$
A fundamental problem in several areas of algebra and geometry is to determine the
least degree of a nonzero homogeneous form in $k[\pr^n]$ vanishing to order at least $m$
on a union of linear subvarieties; i.e., to determine the numbers $\alpha(I^{(m)})$.

By a \emph{fat flat} we mean a scheme defined by $I^{(m)}$,
where $I$ is the radical ideal of a linear subspace $L$ of $\pr^n$ (note in this case that $I^{(m)}=I^m$).
We say that a form $f$ vanishes to order at least $m$ along $L$, if $f\in I^{(m)}$.
We begin by determining the number of conditions such vanishing imposes on forms
of degree $t$.
   The formula without proof appears in \cite[Section 2.2]{Boc05} and in a slightly
   more general setting again without proof in \cite[Corollary 3.4]{refDer}.
   At the risk of belaboring the obvious, we include the proof and show additionally that the expected dimension
   is the actual dimension in degrees $t \geq m$.

\begin{lemma}\label{lem:NumCond}
Let $L$ be a linear subspace of $\P^n$ of dimension $r<n$.
Let $t\geq m$ be positive integers. Then
vanishing to order at least $m$ along $L$ imposes exactly
$$c_{n,r,m,t}=\sum_{0\leq i<m} \binom{t-i+r}{r}\binom{i+n-r-1}{n-r-1}$$
linearly independent conditions on forms of degree $t$ (i.e., $\dim_k ((k[\pr^n])_t/(I(L)^{(m)})_t)=c_{n,r,m,t}$).
Moreover, when $m=t$, we have $c_{n,r,m,m}=\binom{m+n}{n}-\binom{m+n-r-1}{n-r-1}$.
\end{lemma}

\proof
After a change of coordinates,
we may assume that $I(L)\subset k[\pr^n]=k[x_0,\ldots,x_n]$ is generated by
$x_{r+1},\ldots,x_n$.
A monomial $\mu=x_0^{m_0}\cdots x_n^{m_n}$ of degree $t$ vanishes to order less than $m$ on $L$
if and only if $\mu\not\in I^m$; i.e., if and only if $m_{r+1}+\cdots+m_n<m$, and hence if and only if
$m_0+\cdots+m_r\geq t-m$. Thus $\mu\not\in I^m$ if and only if $\mu=\mu_{1i}\mu_{2i}$
for some $0\leq i<m$ where $\mu_{1i}=x_0^{m_0}\cdots x_r^{m_r}$ has degree $t-i$
and $\mu_{2i}=x_{r+1}^{m_{r+1}}\cdots x_n^{m_n}$ has degree $i$.
Since there are $\binom{t-i+r}{r}$ such monomials $\mu_{1i}$, and
$\binom{i+n-r-1}{n-r-1}$ such monomials $\mu_{2i}$, there are $\binom{t-i+r}{r}\binom{i+n-r-1}{n-r-1}$
monomials of the form $\mu_{1i}\mu_{2i}$ and the claim follows by summing over $i$.
When $m=t$, we are counting all monomials of degree $m$ (of which there are $\binom{m+n}{n}$)
except those of degree $m$ in $x_{r+1},\ldots, x_n$ (of which there are $\binom{m+n-r-1}{n-r-1}$),
from which our second claim follows immediately.
\endproof

\begin{remark}
A recursive approach can be used to find $c_{n,r,m,t}$.
It is easy to check that the scheme $mL$ defined by the ideal
$I(mL)=I(L)^{(m)}$, where $L$ is an $r$-plane $L\subsetneq \P^n$,
is arithmetically Cohen-Macaulay (alternatively,
see, for example, \cite[Theorem 3.1]{refGHM2}). This means that
$\Delta^r H_{n,r,m}(t)=H_{n-r,0,m}(t)$, where
$H_{i,j,m}$ is the Hilbert function of a $j$-plane in $\P^i$ of multiplicity $m$.
(So, for example, $H_{n,r,m}(t)=\dim_k (k[\P^n]/I(mL))_t$. Also,
given any function $f:{\bf N}\to {\bf N}=\{0,1,2,\ldots\}$,
we define the \emph{difference operator} $\Delta f$ as
$\Delta f(t)=f(t)-f(t-1)$ for $t>0$ and $\Delta f(0)=f(0)$.)
Since $H_{n-r,0,m}(t)=\min\{\binom{t+n-r}{n},\binom{m+n-r-1}{n-r}\}$,
one can use iterated summation to recover $H_{n,r,m}$ from $H_{n-r,0,m}$.
But $c_{n,r,m,t}=H_{n,r,m}(t)$ for $t\geq m$, so we also recover $c_{n,r,m,t}$.
For example, regard $H_{n,r,m}$ as the sequence
$\{H_{n,r,m}(0), H_{n,r,m}(1), H_{n,r,m}(2),\ldots, H_{n,r,m}(t),\ldots\}$. Taking $m=4$,
$r=2$ and $n=4$, we have
$H_{2,0,4}=\{1,3,6,10,10,10,\ldots, 10,\ldots\}$, so
$H_{3,1,4}=\{1,4,10,20,30,40,\ldots, 10(t-1),\ldots\}$ and
$H_{4,2,4}=\{1,5,15,35,65,105,\ldots, 5t^2-5t+5,\ldots\}$.
Thus $c_{2,0,4,t}=10$, $c_{3,1,4,t}=10t-10$ and $c_{4,2,4,t}=5t^2-5t+5$.
\end{remark}

For $t\geq m-r-1$, we note that the expression $\binom{t+n}{n}-sc_{n,r,m,t}$ is a
polynomial in $t$ of degree $n$, which we denote by $P_{n,r,s,m}(t)$.
Similarly, if ${\bf v}=(m_1,\ldots,m_s)$ is a positive integer vector,
then $\binom{t+n}{n}-\sum_ic_{n,r,m_i,t}$ is a polynomial of degree $n$ in $t$ for
$t\geq\max_i(m_i-r-1)$. We denote this polynomial by $P_{n,r,{\bf v}}(t)$.

\begin{lemma}\label{lem:HilbPolyn}
Let $L_1,\ldots,L_s$ be disjoint linear subspaces of $\P^n$ of dimension $r<n$
(note that disjointness implies $n\geq2r+1$ when $s>1$), let
${\bf v}=(m_1,\ldots,m_s)$ be a positive integer vector and let $m$ be the maximum entry.
Then $P_{n,r,{\bf v}}(t)$ is the Hilbert polynomial of $I=\cap_i I(L_i)^{m_i}$;
i.e., $\dim (I_t) = P_{n,r,{\bf v}}(t)$ for $t\gg0$.
Moreover, if $P_{n,r,{\bf v}}(t)>0$ and $t\geq m\geq 1$, then
$\dim (I_t)>0$ and hence $\alpha(I)\leq t$.
\end{lemma}

\proof
Let $Z$ be the scheme defined by $I$. We denote the sheafification of $I$ by $\cali$,
so we have
$$\dim_k (I_t)=h^0(\pr^n,\cali(t)).$$
We also have an exact sequence of sheaves
\begin{equation}\label{eq:ses}
0\to \cali(t)\to\calo_{\P^n}(t)\to \calo_{Z}(t)\to 0.
\end{equation}
Because the $L_i$ are disjoint, we have $\calo_{Z}(t)=\oplus_i\calo_{m_iL_i}(t)$,
where $m_iL_i$ is the scheme defined by $I(L_i)^{m_i}$.
If we denote the sheaf of ideals of $L_i$ by $\cali_i$,
we have
\begin{equation}\label{eq:ses2}
0\to \cali_i^{m_i}(t)\to\calo_{\P^n}(t)\to \calo_{m_iL_i}(t)\to 0.
\end{equation}
By Serre vanishing we know $h^1(\pr^n,\cali_i^{m_i}(t))=0$ for $t\gg0$, in which case
\eqref{eq:ses2} is exact on global sections, so
$$h^0(m_iL_i,\calo_{m_iL_i}(t))=h^0(\pr^n,\calo_{\P^n}(t))-h^0(\pr^n,\cali_i^{m_i}(t))=
\dim_k(k[\pr^n])_t-\dim_k(I(L_i)^{m_i})_t $$
and by Lemma \ref{lem:NumCond}, this last quantity
is equal to $c_{n,r,m_i,t}$ for $t\geq m_i\geq 1$.
Thus for $t\gg0$ we have
$$\dim (I_t) = h^0(\pr^n,\cali^{m_i}(t))=h^0(\pr^n,\calo_{\P^n}(t))-h^0(Z,\calo_{Z}(t))=
\binom{t+n}{n}-\sum_ic_{n,r,m_i,t}=P_{n,r,{\bf v}}(t).$$

The last statement is linear algebra
(when $t\geq m$, elements of $I_t$ are by Lemma \ref{lem:NumCond} solutions to
$\sum_ic_{n,r,m_i,t}$ homogeneous linear equations on the vector space
$(k[\pr^n])_t$ of forms of degree $t$, which has dimension $\binom{t+n}{n}$;
thus when $P_{n,r,{\bf v}}(t)=\binom{t+n}{n}-\sum_ic_{n,r,m_i,t}>0$,
we must have a nonzero element of $I_t$). Note that the assumption
   $t\geq m$ cannot be relaxed, as it might happen that
   $P_{n,r,{\bf v}}(t)>0$ for some $t\in(0,m)$ whereas obviously $I_t=0$ in this range.
Indeed,
take for example $s$ general lines with $t=1$, $n=3$ and $m\gg0$.
Then $P_{n,1,s,m}(1)>0$ (see display \eqref{eq:HPforLines}) but clearly $\alpha(I^{(m)})>1$.
\endproof

Our main focus will be on the case of symbolic powers of ideals $I=I(L_1\cup\cdots\cap L_s)$
of disjoint $r$-planes in $\pr^n$. In this case, taking ${\bf v}=(m_1,\ldots,m_s)$ with $m_i=m$
for all $i$, we have $P_{n,r,{\bf v}}=P_{n,r,s,m}$, and these are the Hilbert polynomial of $I^{(m)}$.

   In the view of the last statement in Lemma \ref{lem:HilbPolyn} it is natural
   to introduce the following quantity
\begin{equation}
   \enrs:=\inf\left\{\frac{t}{m}:\, t\geq m\geq 1,\, P_{n,r,s,m}(t)>0 \right\},
\end{equation}
   which we call the \emph{expected Waldschmidt constant} of $I$.

Substitute $t=m\tau$ into $P_{n,r,s,m}(t)$ and regard $P_{n,r,s,m}(m\tau)$ as a
polynomial in $m$, whose degree is also $n$. We denote its
leading term by $\Lambda_{n,r,s}(\tau)$ and regard it in turn as a polynomial in $\tau$, again of degree $n$.
The reason $\Lambda_{n,r,s}(\tau)$ is of interest is because of the next theorem and the following
two fundamental conjectures.

\begin{theorem}\label{thm:zeroes of Lambda}
Let $n,r,s$ be integers with $n\geq 2r+1$, $r\geq0$ and $s\geq1$.
Let $I$ be the ideal of $s$ disjoint $r$-planes in $\pr^n$.
Then the polynomial $\Lambda_{n,r,s}(\tau)$ has a single real root bigger than or equal to 1.
Moreover, if we denote this largest real root by $g_{n,r,s}$, then
$$\gamma(I)\leq \enrs \leq g_{n,r,s}.$$
\end{theorem}

The proof is immediate from Lemmas \ref{lem:zeroes of Lambda} and \ref{lem:inequality c}.

\begin{remark}\label{Nagrem}
The bounds given by Theorem \ref{thm:zeroes of Lambda} are most interesting when
$I$ is the ideal of $s$ generic $r$-planes in $\pr^n$, since $\gamma(I)$ will be larger
the more general the $r$-planes are.
Even in the generic case both upper bounds on $\gamma(I)$ can be useful, since $g_{n,r,s}$
is easier to compute than $\enrs$
but is also sometimes strictly larger, while $\enrs$
in turn is often easier to compute than $\gamma(I)$.

The case $n=2$ and $r=0$ is rather special, since for $s<10$,
$\gamma(I)=\enrs$
with $\enrs < g_{n,r,s}$
unless $s=1,2,4$ or 9, in which case $\gamma(I)=g_{n,r,s}$. For $s\geq 10$, a famous
conjecture of Nagata \cite{refN} asserts that $\gamma(I)=\sqrt{s}$, and in fact
$\sqrt{s}=g_{2,0,s}$.
Nagata's conjecture has been extended to $n>2$ (but still $r=0$) by
Iarrobino \cite{refI}. Iarrobino's conjecture asserts for each $n>2$ and for $s\gg0$
generic points that $\gamma(I)=\sqrt[n]{s}$, and again it turns out that
$\sqrt[n]{s}=g_{n,0,s}$.

However, unlike the case $n=2$, for $n>2$
there are values of $s$ for which both inequalities of Theorem \ref{thm:zeroes of Lambda}
can be strict. For example, for $s=4$ generic points of $\pr^3$ (so $n=3$, $r=0$ and $s=4$),
we have $\gamma(I)=4/3$ by Proposition \ref{lem:cremona},
$e_{3,0,4}=3/2$ by Example \ref{infptsinP3},
and $g_{3,0,4}=\sqrt[3]{4}\approx1.587$.
Likewise, for $n=3$, $r=1$ and $s=6$ generic lines in $\pr^3$, we have
$\gamma(I)\leq42/11\approx3.8181$ by Remark \ref{6genericlines},
$e_{3,1,6}=27/7\approx3.8571$
by Example \ref{inflinesinP3},
while $g_{3,0,4}$ is the largest real root of $\tau^3-18\tau+12=0$, which is approximately 3.8587.
However, for $r=0,1$ and $s\gg0$ and all $n\geq 2r+1$, the second inequality is in fact
an equality. It seems reasonable to expect that this holds for all $r$; this is our Conjecture {\conjB} below.
The Nagata/Iarrobino conjecture, together with our Proposition \ref{prop:lines in pn} showing
$\gamma(I)=g_{n,1,s}=n-1$ for $s=(n-1)^{n-2}$ generic lines in $\pr^n$,
motivates Conjecture {\conjC}. Of course, it is a priori possible that
the first inequality holds for $s\gg0$ generic $r$-planes, even if the second does not.
This motivates our Conjecture {\conjA}.
\end{remark}

\begin{varthm*}{Conjecture A}
Let $n>0$ and $r\geq0$ be integers with $n\geq 2r+1$. Let $J_{n,r,s}$ be the ideal
of $s\gg 0$ generic $r$-planes in $\P^n$. Then
$$\gamma(J_{n,r,s})=\enrs.$$
\end{varthm*}

\begin{varthm*}{Conjecture B}
If $n\geq2r+1$, $r\geq1$ and $s\gg0$ are integers, then
$$\enrs=g_{n,r,s}.$$
\end{varthm*}

By Theorem \ref{thm:zeroes of Lambda}(c) we see that
Conjectures {\conjA} and {\conjB} are equivalent to the following conjecture,
which if true gives a numerical approach for computing $\gamma(J_{n,r,s})$.
\begin{varthm*}{Conjecture C}
With the same hypotheses as in Conjecture {\conjA}, we have
$$\gamma(J_{n,r,s})=g_{n,r,s}.$$
\end{varthm*}

\begin{remark}\label{rem:conjs}
It is elementary to check that Conjecture {\conjC} holds for $n=1$
(in which case we have $r=0$, since $n\geq 2r+1$).
As noted above, for $n=2$ and $r=0$, Conjecture {\conjC} is a less explicit
version of the famous Nagata Conjecture \cite{refN}
(Nagata stipulates $s\geq 9$, rather than merely $s\gg0$).
For $n\geq3$ and $r=0$, Conjecture {\conjC} is again a less explicit
version of Iarrobino's generalization of Nagata's conjecture \cite{refI}.
Nagata showed that his conjecture holds if the number of points is a perfect square \cite{refN},
and Evain has shown that Iarrobino's generalized conjecture holds if $s$ is an $n$th power \cite{Evain}.
When $n=3$ and $r=1$, Conjecture {\conjC} is Conjecture 5.5 of \cite{refGHVT}.
\end{remark}

In order to prove Theorem \ref{thm:zeroes of Lambda}, we will need some preliminary results.

\begin{lemma}\label{lem:inequality c}
Given integers $n>0$, $r\geq0$ and $s\geq1$ with $n\geq 2r+1$ and disjoint $r$-dimensional
linear subspaces $L_1,\ldots,L_s\subset \pr^n$ whose union has ideal $I$,
we have
$$\gamma(I)\leq \enrs.$$
Moreover, if $\Lambda_{n,r,s}(\tau)>0$ for all $\tau>g_{n,r,s}$ and if $g_{n,r,s}\geq 1$, then
$$\enrs\leq g_{n,r,s}.$$
\end{lemma}

\begin{proof}
By Lemma \ref{lem:HilbPolyn}, $\alpha(I^{(m)})/m\leq t/m$ if $P_{n,r,s,m}(t)>0$ for $t\geq m\geq 1$.
Thus $\gamma(I)\leq \inf\left\{\frac{t}{m}:\, t\geq m\geq 1,\, P_{n,r,s,m}(t)>0 \right\}$
now follows from the definition of the Waldschmidt constant.

For the other inequality,
let $\tau'$ be any fixed rational number bigger than $g_{n,r,s}$.
By assumption, $\Lambda_{n,r,s}(\tau')>0$,
so, for $t=m\tau'$, the leading term of the Hilbert polynomial
$P_{n,r,s,m}(t)=P_{n,r,s,m}(\tau'm)$ (viewed as a polynomial in $m$)
is positive.
Hence the polynomial itself is also positive for large $m$.
Since we are assuming $g_{n,r,s}\geq 1$, we see that
$\tau'$ is in the set $\left\{\frac{t}{m}:\, t\geq m\geq 1,\, P_{n,r,s,m}(t)>0 \right\}$.
This works for all $\tau'>g_{n,r,s}$, hence
$$\enrs=\inf\left\{\frac{t}{m}:\, t\geq m\geq 1,\, P_{n,r,s,m}(t)>0 \right\}\leq g_{n,r,s}.$$
\end{proof}

Expressions for the Hilbert polynomial $P_{n,r,s,m}(t)$
of the union of $s$ disjoint $r$-planes of multiplicity $m$ in $\pr^n$
are complicated (even when $m=1$; see for example
\cite[Corollary 3.4]{refDer}). Thus even though
$P_{n,r,s,1}(t)=\binom{t+n}{n}-sc_{n,r,m,t}$,
together with Lemma \ref{lem:NumCond}, in principle gives an explicit
formula for $P_{n,r,s,1}(t)$, it would not be easy to
extract $\Lambda_{n,r,s}(\tau)$ from $P_{n,r,s,m}(t)$ directly.
It is, nevertheless, possible to give an explicit expression for $\Lambda_{n,r,s}(\tau)$ in all cases.
We do this in Proposition \ref{lem:Lambdan0s}, as an application of
our next two results. We will need the following notation.
Given any function $f(t)$, we define the \emph{first difference} $\Delta f$ as
$\Delta f(t)=f(t)-f(t-1)$.

\begin{lemma}\label{lem:Delta}
For $n\geq 2r+1$ and $r>0$, we have
$$\Delta P_{n,r,s,m}(t)=P_{n-1,r-1,s,m}(t).$$
\end{lemma}

\begin{proof}
Since $n\geq 2r+1$, we can choose $s$ disjoint linear spaces $L_1,\ldots, L_s\subset \pr^n$,
each of dimension $r$. Let $I_{n,r,s}=I(L_1\cup\cdots\cup L_s)\subset k[\pr^n]$ be the ideal of their union.
Take a general hyperplane $H$. We then
have the hyperplane section $H\cap(mL_1\cup\cdots\cup mL_s)$
of $mL_1\cup\cdots\cup mL_s$.
The linear subspaces $L_i\cap H$ are again disjoint in $H\simeq\P^{n-1}$.
Let $I_{n-1,r-1,s}$ be the ideal of their union.
The fact that the $L_i$ are disjoint means that the ideal
$I_{n-1,r-1,s}^{(m)}$
is the saturation of $\left(I_{n,r,s}^{(m)}+I(H)\right)/I(H)$ in the ring $k[H]$.
Hence for $t\gg 0$
$$\left(I_{n-1,r-1,s}^{(m)}\right)_t=\left(\left(I_{n,r,s}^{(m)}+I(H)\right)/I(H)\right)_t.$$
But
$$\left(I_{n,r,s}^{(m)}+I(H)\right)/I(H)\cong I_{n,r,s}^{(m)}/\left(I_{n,r,s}^{(m)}\cap I(H)\right)
\cong I_{n,r,s}^{(m)}/\left(F\cdot I_{n,r,s}^{(m)}\right),$$
where $F$ is the linear form defining $H$. Thus
$$\dim (I_{n-1,r-1,s}^{(m)})_t=
\dim \left(I_{n,r,s}^{(m)}/F\cdot I_{n,r,s}^{(m)}\right)_t=
\dim (I_{n,r,s}^{(m)})_t-\dim(I_{n,r,s}^{(m)})_{t-1},$$
which for
$t\gg0$ gives $P_{n-1,r-1,s,m}(t)=\Delta P_{n,r,s,m}(t)$.
\end{proof}

Lemma \ref{lem:Delta} implies that the polynomials $\Lambda_{n,r,s}$ are linked
by a series of linear differential equations.

\begin{corollary}\label{cor:diff eq tower}
For $n\geq 2r+1$ and $r>0$, we have $\Lambda_{n,r,s}(1)=(1-s)/n!$ and
$$\Lambda_{n-1,r-1,s}(\tau)=\frac{d\ \Lambda_{n,r,s}(\tau)}{d\tau}.$$
\end{corollary}

\begin{proof}
Since $P_{n-1,r-1,s,m}(m\tau)=P_{n,r,s,m}(m\tau)-P_{n,r,s,m}(m\tau-1)$
has degree $m^{n-1}$ in $m$, the leading coefficient $\Lambda_{n-1,r-1,s}(\tau)$ of
$P_{n-1,r-1,s,m}(m\tau)$ can be obtained by taking a limit
\[
\begin{split}
\Lambda_{n-1,r-1,s}(\tau)&=\lim_{m\to\infty}\frac{P_{n-1,r-1,s,m}(m\tau)}{m^{n-1}}=
\lim_{m\to\infty}\frac{P_{n,r,s,m}(m\tau)-P_{n,r,s,m}(m\tau-1)}{m^{n-1}}\\
&=\lim_{m\to\infty}\frac{m^n(\Lambda_{n,r,s}(\tau)-\Lambda_{n,r,s}(\tau-(1/m)))}{m^{n-1}}
=\frac{d\ \Lambda_{n,r,s}(\tau)}{d\tau}.
\end{split}
\]

Applying the formula for $c_{n,r,m,m}$ in Lemma \ref{lem:NumCond} we see that
$$P_{n,r,s,m}(m)=\binom{m+n}{n}-sc_{n,r,m,m}=
\binom{m+n}{n}-s\Big(\binom{m+n}{n}-\binom{m+n-r-1}{n-r-1}\Big),$$
hence $\Lambda_{n,r,s}(1)=\frac{1-s}{n!}$.
\end{proof}

\begin{proposition}\label{lem:Lambdan0s}
Let $n\geq 2r+1$, $r\geq0$ and $s\geq 1$. Then
$$\Lambda_{n,r,s}(\tau)=\frac{1}{n!}\Big(\tau^n-s\Big(\sum_{j=0}^r\binom{n}{j}(\tau-1)^j\Big)\Big).$$
In particular,
$$\Lambda_{n,0,s}(\tau)=(\tau^n-s)/n!\;\mbox{ and }\; g_{n,0,s}=\sqrt[n]{s}.$$
\end{proposition}

\proof
It is easy to see and well known that
the Hilbert polynomial of $s$ points of multiplicity $m$ in $\pr^n$ is $s\binom{m+n-1}{n}$, and hence the
Hilbert polynomial of the corresponding ideal $I_{n,0,s}$ is
$$P_{n,0,s,m}(t)=\binom{t+n}{n}-s\binom{m+n-1}{n}=\frac{1}{n!}\big((t+n)\cdots(t+1)-s(m+n-1)\cdots(m+1)m\big).$$
Substituting $m\tau=t$ gives (for appropriate values of the $a_i$ and $b_i$ for $1\leq i\leq n-1$, all of which are positive)
\begin{equation}\label{eq:Pn0sm}
\begin{split}
P_{n,0,s,m}(m\tau)=&\big((m\tau+n)\cdots(m\tau+1)-s(m+n-1)\cdots(m+1)m\big)/n!\\
                                =&\big(m^n(\tau^n-s)+m^{n-1}(a_{n-1}\tau^{n-1}-b_{n-1}s)+\cdots+m(a_1\tau-b_1s)+n!\big)/n!
\end{split}
\end{equation}
Thus indeed $\Lambda_{n,0,s}(\tau)=(\tau^n-s)/n!$ and $g_{n,0,s}=\sqrt[n]{s}$.
For the rest, start with $\Lambda_{n-r,0,s}(\tau)=(\tau^{n-r}-s)/(n-r)!$, and take
$r$ antiderivatives using Corollary \ref{cor:diff eq tower}, to obtain the result.
\endproof

\begin{lemma}\label{lem:zeroes of Lambda}
Let $n,r,s$ be integers with $n\geq 2r+1$, $r\geq0$ and $s\geq1$.
Let $I$ be the ideal of $s$ disjoint $r$-planes in $\pr^n$.
Then the polynomial $\Lambda_{n,r,s}(\tau)$ has a single real root bigger than or equal to 1.
Moreover, if we denote this largest real root by $g_{n,r,s}$, then:
\begin{itemize}
\item[(a)] $g_{n,r,s}>g_{n-1,r-1,s}$ for $r>0$ with $s>1$, while $g_{n,r,1}=1<g_{n,r,s}$ for $r\geq 0$ and $s>1$; and
\item[(b)] $\Lambda_{n,r,s}(\tau)<0$ for $1\leq \tau<g_{n,r,s}$ with
$\Lambda_{n,r,s}(\tau)>0$ for $\tau>g_{n,r,s}$.
\end{itemize}
\end{lemma}

\begin{proof}
We first prove (a) and (b) for the case $s=1$.
By Proposition \ref{lem:Lambdan0s},
$\Lambda_{n,r,s}(\tau)=\frac{1}{n!}\Big(\tau^n-s\Big(\sum_{j=0}^r\binom{n}{j}(\tau-1)^j\Big)\Big)$,
so $\Lambda_{n,r,1}(1)=0$. Moreover, by Corollary \ref{cor:diff eq tower},
$d(\Lambda_{n-r+1,1,1}(\tau))/d\tau=\Lambda_{n-r,0,1}(\tau)=(\tau^n-1)/n!>0$
for $\tau>1$, so $\Lambda_{n-r+1,1,1}(\tau)$ is increasing for $\tau\geq 1$,
hence positive for $\tau>1$.
The same argument applied to $d(\Lambda_{n-r+2,2,1}(\tau))/d\tau=\Lambda_{n-r+1,1,1}(\tau)$
shows that $\Lambda_{n-r+2,2,1}(\tau)$ is increasing for $\tau\geq 1$ hence positive for $\tau>1$.
Continuing in this way we eventually obtain that
$\Lambda_{n,r,1}(\tau)$ is strictly increasing for $\tau\geq1$, hence positive for $\tau>1$.
Therefore $\tau=1$ is the largest real root, so
$g_{n,r,1}=1$ and now (b) follows when $s=1$.

Now assume $s>1$. Then it is easy to check that $\Lambda_{n-r+i,i,s}(1)<0$ for all $0\leq i\leq r$. But
as a polynomial in $\tau$, $\Lambda_{n-r+i,i,s}(\tau)$ has positive leading coefficient
so it has a root bigger than 1. But $\Lambda_{n-r,0,s}(\tau)=(\tau^n-s)/n!$
is positive for $\tau>1$, so $\Lambda_{n-r+1,1,s}(\tau)$, being an antiderivative,
is strictly increasing for $\tau>1$, so it has exactly one root bigger than 1, hence
$g_{n-r+1,1,s}>1$. In particular, $\Lambda_{n-r+1,1,s}(\tau)<0$ for $1\leq \tau<g_{n-r+1,1,s}$
and $\Lambda_{n-r+1,1,s}(\tau)>0$ for $\tau>g_{n-r+1,1,s}$.
Since $\Lambda_{n-r+1,1,s}(\tau)$ is an antiderivative for
$\Lambda_{n-r+2,2,s}(\tau)$, we now see that $\Lambda_{n-r+2,2,s}(\tau)$
is decreasing and hence negative for $1\leq \tau\leq g_{n-r+1,1,s}$.
Since $\Lambda_{n-r+2,2,s}(\tau)$ has positive leading coefficient, it has a root
bigger than $g_{n-r+1,1,s}$, hence $g_{n-r+2,2,s}>g_{n-r+1,1,s}$. Because
$\Lambda_{n-r+1,1,s}(\tau)>0$ for $\tau>g_{n-r+1,1,s}$, we see that
$\Lambda_{n-r+2,2,s}(\tau)$ is strictly increasing for $\tau\geq g_{n-r+1,1,s}$,
hence $\Lambda_{n-r+2,2,s}(\tau)$ has exactly one root bigger than $g_{n-r+1,1,s}$,
and we see $\Lambda_{n-r+2,2,s}(\tau)<0$ for $1\leq \tau<g_{n-r+2,2,s}$
and $\Lambda_{n-r+2,2,s}(\tau)>0$ for $\tau>g_{n-r+2,2,s}$.
The argument continues in this way for $\Lambda_{n-r+i,i,s}(\tau)$ for $i$ up to $r$,
which proves (a) and (b).
\end{proof}

\begin{remark}\label{rmk:nef cone}
The polynomials $\Lambda_{n,r,s}(\tau)$ can be found as intersections
of certain classes on the blow up of $\pr^n$. Specifically, let $X$ be the blow up of
$\pr^n$ along $s$ disjoint $r$-planes, $P_1,\ldots,P_s$.
Let $H$ be the hyperplane class pulled back to $X$, let $E_i$ be the exceptional locus
of $P_i$ and let $E=E_1+\cdots+E_s$.
Then $(\tau H-E)^n=n!\Lambda_{n,r,s}(\tau)$ (see Corollary \ref{cor:IntThry=Lambda}).
This implies that $m(\tau H-E)$ is big when $\tau>g_{n,r,s}$ is rational and $m$
is sufficiently large and sufficiently divisible.\footnote{To see this, let
$Z\subset \pr^n$ be the scheme theoretic union of the $P_i$, and let $mZ$ the scheme defined by
$I(Z)^{(m)}$. Then, as in the proof of Lemma \ref{lem:HilbPolyn} and using the fact that
$g_{n,r,s}\geq 1$ by Theorem \ref{thm:zeroes of Lambda}, we have
$h^0(X,\calo_X(m(\tau H-E)))=\dim I(mZ)_{m\tau}\geq P_{n,r,s,m}(m\tau)$, but
as a polynomial in $m$, $P_{n,r,s,m}(m\tau)$ has leading coefficient
$\Lambda_{n,r,s}(\tau)>0$, and thus $h^0(X,\calo_X(m(\tau H-E)))$ has order of growth
$m^n$.}
Note that $(\tau H-E)^n>0$ by itself has no immediate consequences for
the bigness of multiples of $\tau H-E$ when $n\geq 3$ and $r>0$.
\end{remark}

\section{Conjecture {\conjB} for $r\leq 1$}
In this section we assume $r\leq1$. We begin by verifying Conjecture {\conjB} for $r=0$.

\begin{proposition}
Conjecture {\conjB} holds for $r=0$ for all $n\geq 1$.
\end{proposition}

\begin{proof}
For $\tau=t/m$, note that \eqref{eq:Pn0sm} implies that $P_{n,0,s,m}(t)>0$ is equivalent to
$$m^n(\tau^n-s)+m^{n-1}(a_{n-1}\tau^{n-1}-b_{n-1}s)+\cdots+m(a_1\tau-b_1s)\geq0,$$
but for $s\gg0$ and $0\leq \tau\leq g_{n,0,s}$, we have $a_i\tau^i-b_is\leq a_ig_{n,0,s}^i-b_is<0$
(and hence $P_{n,0,s,m}(t)\leq 0$ for $0\leq \tau\leq g_{n,0,s}$ with $s\gg0$).
On the other hand, for $\tau>g_{n,0,s}$, we have $\tau^n-s>0$ and
so for $m\gg0$ we have $P_{n,0,s,m}(t)=P_{n,0,s,m}(m\tau)>0$.
Since there are rationals $\tau=t/m$ with $m\gg0$
arbitrarily close to but bigger than $g_{n,0,s}$, the result follows.
\end{proof}

We now consider the case that $r=1$.
The fact that Conjecture {\conjB} holds for lines for $n = 3$ was shown by \cite{refGHVT}. We will
extend this to all $n\geq 3$.
The Hilbert polynomial of $s$ disjoint lines of multiplicity $m$ in $\pr^n$ is
$s((t+1)\binom{m+n-2}{n-1}-(n-1)\binom{m+n-2}{n})$, and hence the
Hilbert polynomial of the corresponding ideal is
\begin{equation}\label{eq:HPforLines}
P_{n,1,s,m}(t)=\binom{t+n}{n}-s\Bigg((t+1)\binom{m+n-2}{n-1}-(n-1)\binom{m+n-2}{n}\Bigg)
\end{equation}
and by Proposition \ref{lem:Lambdan0s}
\begin{equation}\label{eq:LambdaforLines}
\Lambda_{n,1,s}(\tau)=\frac{\tau^n-ns\tau+(n-1)s}{n!}.
\end{equation}

We now verify Conjecture {\conjB} for $r=1$.

\begin{theorem}\label{Thm:ConjB for r=1}
Conjecture {\conjB} holds for  $r=1$ for all $n\geq3$.
\end{theorem}

\begin{proof}
By Lemma \ref{lem:zeroes of Lambda} we have
$\inf\left\{\frac{t}{m}:\, t\geq m,\, P_{n,1,s,m}(t)>0 \right\}\leq g_{n,1,s}$;
we need to justify the reverse inequality when $s\gg0$.
To do so, note that $\binom{t+n}{n}=\frac{(t+n)\cdots(t+1)}{n!}$, hence substituting
$t=m\tau$ gives $\frac{(m\tau+n)\cdots (m\tau+1)}{n!}=1+\sum_{i=1}^na_im^i\tau^i$,
for appropriate coefficients $a_i>0$. Similarly, substituting $t=m\tau$
into $(t+1)\binom{m+n-2}{n-1}-(n-1)\binom{m+n-2}{n}$ gives
\[
\begin{split}
&(m\tau+1)\binom{m+n-2}{n-1}-(n-1)\binom{m+n-2}{n}\\
&=\frac{(m+n-2)\cdots (m+1)m}{n!}\Bigg(n(m\tau+1)-(n-1)(m-1)\Bigg)\\
&=\frac{(m+n-2)\cdots (m+1)m}{n!}\Bigg((n(\tau-1)+1)m+(2n-1)\Bigg)\\
&=\sum_{i=2}^nb_i(\tau-1)m^i+\sum_{j=1}^nd_jm^j
\end{split}
\]
for appropriate positive coefficients $b_i$ and $d_j$ (except we define $b_1=0$). Therefore,
from equation \eqref{eq:HPforLines} we have
$P_{n,1,s,m}(m\tau)=1+\sum_{i=1}^{n-1}(a_i\tau^i-s(b_i(\tau-1)+d_i))m^i+\Lambda_{n,1,s}(\tau)m^n$.
For $s>1$, it follows by Lemma \ref{lem:zeroes of Lambda}(b) that
$n!\Lambda_{n,1,s}(\tau)=\tau^n-ns\tau+(n-1)s$ is negative for $\tau\in [1,g_{n,1,s})$
and positive for $\tau>g_{n,1,s}$. Since $\tau^n-ns\tau+(n-1)s>0$
for $\tau=\sqrt[n-1]{ns}$, we see $g_{n,1,s}<\sqrt[n-1]{ns}$.

Suppose we check that each coefficient $a_i\tau^i-s(b_i(\tau-1)+d_i)$, $1\leq i\leq n-1$, is negative
on the interval $[1,\sqrt[n-1]{ns}]$ for $s\gg0$. Then for $s\gg0$ and any integers $t,m\geq 1$ such that
$\tau=t/m\leq g_{n,1,s}$ we would have
$$P_{n,1,s,m}(t)=P_{n,1,s,m}(m\tau)=1+\sum_{i=1}^{n-1}(a_i\tau^i-s(b_i(\tau-1)+d_i))m^i+\Lambda_{n,1,s}(\tau)m^n<1$$
and hence $P_{n,1,s,m}(t)\leq 0$. Therefore, it would follow that
$\inf\left\{\frac{t}{m}:\, t\geq m,\, P_{n,1,s,m}(t)>0 \right\}\geq g_{n,1,s}$, as we wanted to show.
But $a_i\tau^i-s(b_i(\tau-1)+d_i)$ is concave up on $\tau\geq1$,
and for $s\gg0$ we have $a_i\tau^i-s(b_i(\tau-1)+d_i)<0$ at $\tau=1$, so for $s\gg0$,
we see $a_i\tau^i-s(b_i(\tau-1)+d_i)$ has a single root on the interval $\tau\geq 1$.
Thus, to show $a_i\tau^i-s(b_i(\tau-1)+d_i)<0$ on the interval $[1,\sqrt[n-1]{ns}]$ when $s\gg0$,
it suffices to check that $a_i\tau^i-s(b_i(\tau-1)+d_i)<0$ for $\tau=\sqrt[n-1]{ns}$ when $s\gg0$.
But after the substitution $\tau=\sqrt[n-1]{ns}$, $a_i\tau^i-s(b_i(\tau-1)+d_i)$ becomes a polynomial
in $\sqrt[n-1]{s}$ with negative leading coefficient, and hence we will have
$a_i\tau^i-s(b_i(\tau-1)+d_i)<0$ for $\tau=\sqrt[n-1]{ns}$ for $s\gg0$, as desired.
\end{proof}

\section{A series of examples}\label{section4}

   We now compute $\gamma(I)$ for the ideal $I$ of $s=(n-1)^{(n-2)}$ general lines in $\P^n$.

\begin{proposition}\label{prop:lines in pn}
   Let $I=J_{n,1,(n-1)^{(n-2)}}$ be the ideal of $(n-1)^{(n-2)}$ general lines in $\P^n$. Then
   $$\gamma(I)=g_{n,1,(n-1)^{(n-2)}}=n-1.$$
\end{proposition}
\proof
   It is convenient to shift $n$ and consider the ideal $I$ of $n^{(n-1)}$
   general lines in $\P^{n+1}$ for which we claim $\gamma(I)=n$.

   The inequality $\gamma(I)\leq n$ follows readily from Theorem \ref{thm:zeroes of Lambda},
   since it is easy to check directly that $g_{n+1,1,n^{(n-1)}}=n$.

   For the reverse inequality, it suffices to show that
   \begin{equation}\label{eq:sys0}
      h^0(\P^{n+1},\calo_{\P^{n+1}}(nm-1)\otimes\cali^m)=0
   \end{equation}
   for all $m\geq 1$, where $\cali$ is as usual the sheafification of $I$.
   We proceed by induction on $m$. The case $m=1$ follows easily from \cite{HarHir81}.
   We assume that \eqref{eq:sys0} holds for $m-1$ and aim for proving this for $m$.
   By the upper semicontinuity of cohomology functions it is enough to show
   the statement for lines in a special position. To this end we first choose
   $n$ general lines $L_1,\dots,L_n$ in $\P^{n+1}$. On these lines we
   mimic the construction of a rational $n$--fold scroll as in \cite[Exercise 8.26]{Har92}.
   Specifically, for $i=2,\dots n$ let $\alpha_i:L_1\to L_i$ be a general linear isomorphism.
   Let $X\subset\P^{n+1}$ be defined as the closure of the set of projective subspaces
   in $\P^{n+1}$ spanned by points $p,\alpha_2(p),\dots,\alpha_n(p)$.
   All these subspaces are isomorphic to $\P^{n-1}$ if $L_i$ and $\alpha_i$
   are chosen sufficiently general. Alternatively $X$ can be viewed
   as the image of $\P^{n-1}\times \P^1$ embedded by a morphism $\phi$
   induced by a non-complete linear subsystem of bidegree $(1,1)$.
   In particular $\deg(X)=n$ as in the case of the Segre variety $\Sigma_{n-1,1}$,
   see \cite[Example 18.15]{Har92}.
   
   Now, let $P_1,\dots,P_{n^{n-1}}$ be $n^{n-1}$ general points in $\P^{n-1}$.
   We consider the $n^{n-1}$ lines given as images of $\left\{P_i\right\}\times\P^1$
   under $\phi$. Assume that there exists a divisor $D$ of degree $mn-1$
   vanishing along all these lines with multiplicity at least $m$.
   Restricting $D$ to $\phi(\P^{n-1}\times\left\{x\right\})$ either
   gives a divisor $D_x$ of degree $mn-1$ in $\P^{n-1}$ vanishing to order $m$
   at $n^{n-1}$ general points (the intersection points of 
   $\phi(\P^{n-1}\times\left\{x\right\})$ with $\phi(\left\{P_i\right\}\times\P^1)$)
   or $D$ contains $\phi(\P^{n-1}\times\left\{x\right\})$. The first
   possibility is excluded by a result of Evain \cite[Theorem 3]{Evain} 
   to the effect that $\gamma$ of $a^{n-1}$ general points in $\P^{n-1}$
   equals $a$. Hence the second possibility holds for all points $x\in \P^1$.
   Thus $D$ contains $X$. Applying the induction hypothesis to the divisor $D-X$,
   which has degree $(m-1)n-1$, we get a contradiction. This shows that $\gamma=n$
   for $n^{n-1}$ general lines in $\P^{n+1}$.
\endproof
   Along the same lines computer experiments provide strong evidence
   in favor of the following example. We note that $g_{11,2,729}=3$.
\begin{problem}   
   Show that $\gamma=3$ for $729=9^3$ general planes in $\P^{11}$.
\end{problem}

\section{Appendix 1: Combinatorics}

Lemma \ref{combinatorialLemma2} is used in the proof of Theorem \ref{thm:zeroes of Lambda},
while Lemma \ref{combinatorialLemma1} is used in the proof of Lemma \ref{combinatorialLemma2}.

\begin{lemma}\label{combinatorialLemma1}
Let $a\geq0$ be an integer. Then we have
$$\sum_{0\leq i<m} \binom{i+a}{a}=\binom{m+a}{a+1}.$$
\end{lemma}

\begin{proof}
The formula holds for $m=1$. Now induct on $m$:
$\sum_{0\leq i<m+1} \binom{i+a}{a}=
\binom{m+a}{a}+\sum_{0\leq i<m} \binom{i+a}{a}=
\binom{m+a}{a}+\binom{m+a}{a+1}= \binom{m+1+a}{a+1}$.
\end{proof}

We can now determine $c_{n,1,m,t}$.

\begin{lemma}\label{combinatorialLemma2}
Let $a\geq0$ be an integer. Then we have
\begin{itemize}
\item[(a)] $\displaystyle\sum_{0\leq i<m} i\binom{i+a}{a}=(a+1)\binom{m+a}{a+2}$;
\item[(b)] $\displaystyle\sum_{0\leq i<m} (t-i+1)\binom{i+a}{a} =
(t+1)\binom{m+a}{a+1}-(a+1)\binom{m+a}{a+2}$; and
\item[(c)] $\displaystyle c_{n,1,m,t}=(t+1)\binom{m+n-2}{n-1}-(n-1)\binom{m+n-2}{n}$.
\end{itemize}
\end{lemma}

\begin{proof}
(a) The formula holds for $m=1,2$. Now induct on $m$:
$\sum_{0\leq i<m+1} i\binom{i+a}{a}=(a+1)\binom{m+a}{a+2}+m\binom{m+a}{a}=
(a+1)\big(\binom{m+a}{a+2}+\binom{m+a}{a+1}\big)=(a+1)\binom{m+a+1}{a+2}$.

(b) From (a) and Lemma \ref{combinatorialLemma1}, we have
\[
\begin{split}
\sum_{0\leq i<m} (t-i+1)\binom{i+a}{a} &= \sum_{0\leq i<m} (t+1)\binom{i+a}{a} -\sum_{0\leq i<m} i\binom{i+a}{a}\\
                                                                     &=(t+1)\binom{m+a}{a+1}-(a+1)\binom{m+a}{a+2}.
\end{split}
\]
(c) Apply (b) with $a=n-2$.
\end{proof}

To verify our assertion in Remark \ref{rmk:nef cone} that $(\tau H-E)^n=n!\Lambda_{n,r,s}(\tau)$,
we need a couple of identities.

\begin{lemma}\label{lem:H.E identities}
Let $t,j\geq 0$ be integers. Then
$$\sum_{i=0}^j(-1)^i\binom{t+j}{j-i}\binom{t+i}{i}=0;$$
moreover, if $t\geq 1$, then
$$\sum_{i=0}^j(-1)^i\binom{t+j}{j-i}\binom{t+i-1}{i}=1,$$
hence $-1+\sum_{i=0}^{j-1}(-1)^i\binom{t+j}{j-i}\binom{t+i-1}{i}=(-1)^{j+1}\binom{t+j-1}{j}$.
\end{lemma}

\begin{proof}
If we denote differentiation with respect to $x$ by $D_x$, then
$(j+t)\cdots (j+1)(1-x)^j=D^t_x(1-x)^{t+j}=D_x^t\sum_{i=0}^{t+j}(-1)^i\binom{t+j}{t+j-i}x^i=
\sum_{i=t}^{t+j}(-1)^i\binom{t+j}{t+j-i}t!\binom{i}{t}x^{i-t}=
t!\sum_{i=0}^j(-1)^{t+i}\binom{t+j}{j-i}\binom{t+i}{t}x^i$,
and evaluating at $x=1$ gives $0=(-1)^tt!\sum_{i=0}^j(-1)^i\binom{t+j}{j-i}\binom{t+i}{t}$
and hence $\sum_{i=0}^j(-1)^i\binom{t+j}{j-i}\binom{t+i}{i}=0$.

We prove the second formula by induction on $j$. When $j=0$ it is just $\binom{t+0}{0}\binom{t-1}{0}=1$.
So assume $\sum_{i=0}^{j-1}(-1)^i\binom{t+j-1}{j-i-1}\binom{t+i-1}{i}=1$ holds for some $j\geq 1$.
Adding our first formula (with $t-1$ in place of $t$) to this gives
\[
\begin{split}
1&=\sum_{i=0}^{j-1}(-1)^i\binom{t+j-1}{j-i-1}\binom{t+i-1}{i}+\sum_{i=0}^j(-1)^i\binom{t+j-1}{j-i}\binom{t+i-1}{i}\\
  &=\Bigg(\sum_{i=0}^{j-1}(-1)^i\Bigg[\binom{t+j-1}{j-i-1}+\binom{t+j-1}{j-i}\Bigg]\binom{t+i-1}{i}\Bigg)
  +(-1)^j\binom{t+j-1}{0}\binom{t+j-1}{j}\\
  &=\sum_{i=0}^j(-1)^i\binom{t+j}{j-i}\binom{t+i-1}{i}.
\end{split}
\]
\end{proof}

\begin{corollary}\label{cor:cohring}
Let $X$ be the blow up of $\pr^n$ along an $r$-plane $P$ for some $r<n$,
let $H$ be the hyperplane class pulled back to $X$ and let $E$ be the exceptional locus
of $P$. Then $H^n=1$, $H^jE^{n-j}=(-1)^{n+1-r}\binom{n-j-1}{r-j}$
for $0\leq j\leq r$, and $H^jE^{n-j}=0$ for $r<j<n$.
\end{corollary}

\begin{proof}
It is obvious that $H^n=1$, and
we have $H^{r+1}E=0$ and $(H-E)^{n-r}=0$ by \cite[Corollary 2.5]{Magg00}.
>From $H^{r+1}E=0$ we see that $H^jE^{n-j}=0$ for $r<j<n$.
For $j=r$ we have
$0=H^r(H-E)^{n-r}=H^n+\sum_{l=0}^{n-r}(-1)^{n-r-l}\binom{n-r}{l}H^{r+l}E^{n-r-l}=1+(-1)^{n-r}H^rE^{n-r}$,
so $H^rE^{n-r}=(-1)^{n-r+1}=(-1)^{n+1-r}\binom{n-r-1}{0}$.
Now let $j=r-i$ for some $0<i\leq r$.
Then $H^jE^{n-j}=H^{r-i}E^{n-r+i}$ and we want to verify that
$H^{r-i}E^{n-r+i}=(-1)^{n+1-r}\binom{n-r+i-1}{i}$. By induction we may assume
that $H^{r-l}E^{n-r+l}=(-1)^{n+1-r}\binom{n-r+l-1}{l}$ holds for $0\leq l<i$.
Thus
\[
\begin{split}
0&=H^{r-i}(H-E)^{n-r+i}=H^n+\sum_{l=0}^i(-1)^{n-r+l}\binom{n-r+i}{i-l}H^{r-i}H^{i-l}E^{n-r+l}\\
&=1+\sum_{l=0}^{i-1}(-1)^{l+1}\binom{n-r+i}{i-l}\binom{n-r+l-1}{l}+(-1)^{n-r+i}\binom{n-r+i}{0}H^{r-i}E^{n-r+i},
\end{split}
\]
so
$H^{r-i}E^{n-r+i}=(-1)^{n-r+i}\big(-1+\sum_{l=0}^{i-1}(-1)^l\binom{n-r+i}{i-l}\binom{n-r+l-1}{l}\big)$.
By Lemma \ref{lem:H.E identities} this equals $(-1)^{n-r+1}\binom{n-r+i-1}{i}$, as claimed.
\end{proof}

Now let $X$ be the blow up of
$\pr^n$ along $s$ disjoint $r$-planes, $P_1,\ldots,P_r$.
Let $H$ be the hyperplane class pulled back to $X$, let $E_i$ be the exceptional locus
of $P_i$ and let $E=E_1+\cdots+E_s$.
Clearly, $E_iE_j=0$ for $i\neq j$, and hence by Corollary \ref{cor:cohring}, we have
$H^n=1$, $H^jE^{n-j}=s(-1)^{n+1-r}\binom{n-j-1}{r-j}$
for $0\leq j\leq r$ and $H^jE^{n-j}=0$ for $r<j<n$. One can then compute that
$$(\tau H-E)^n=\tau^n-s\sum_{j=0}^r(-1)^{r-j}\binom{n}{j}\binom{n-j-1}{r-j}\tau^j.$$
By applying Lemma \ref{lem:H.E identities}, we can verify that
this is equal to $n!\Lambda_{n,r,s}(\tau)=\tau^n-s\sum_{j=0}^r\binom{n}{j}(\tau-1)^j$:

\begin{corollary}\label{cor:IntThry=Lambda}
As polynomials in a variable $x$, we have
$$x^n-s\sum_{j=0}^r(-1)^{r-j}\binom{n}{j}\binom{n-j-1}{r-j}x^j=
x^n-s\sum_{j=0}^r\binom{n}{j}(x-1)^j.$$
\end{corollary}

\begin{proof}
Let $F_{n,r,s}(x)$ denote the polynomial $x^n-s\sum_{j=0}^r\binom{n}{j}(x-1)^j$,
and let $G_{n,r,s}(x)$ denote $x^n-s\sum_{j=0}^r(-1)^{r-j}\binom{n}{j}\binom{n-j-1}{r-j}x^j$.
It is easy to see that the derivative satisfies $D_xF_{n,r,s}=nF_{n-1,r-1,s}(x)$, and that
$F_{n,r,s}(1)=1-s$ and $F_{n-r,0,s}(x)=x^{n-r}-s$.
By induction it is enough to show that
$D_xG_{n,r,s}=nG_{n-1,r-1,s}(x)$, $G_{n,r,s}(1)=1-s$, and
$G_{n-r,0,s}(x)=x^{n-r}-s$. This is straightforward, except possibly for checking
$G_{n,r,s}(1)=1-s$, which is obviously equivalent to checking
$\sum_{j=0}^r(-1)^{r-j}\binom{n}{j}\binom{n-j-1}{r-j}=1$.
But $\sum_{j=0}^r(-1)^{r-j}\binom{n}{j}\binom{n-j-1}{r-j}=
\sum_{j=0}^r(-1)^{r-j}\binom{n}{j}\binom{n-j-1}{n-r-1}$, and by reversing the
order if the summation (by replacing $j$ by $r-j$), this becomes
$\sum_{j=0}^r(-1)^j\binom{n}{r-j}\binom{n-r+j-1}{n-r-1}$ and,
which in turn becomes
$\sum_{j=0}^r(-1)^j\binom{n}{r-j}\binom{n-r+j-1}{j}$
by symmetry of binomial coefficients, which is 1 by
Lemma \ref{lem:H.E identities} (take $t=n-r$).
\end{proof}

\renewcommand{\thetheorem}{\thesubsection.\arabic{theorem}}
\setcounter{theorem}{0}

\section{Appendix 2: Calculations}
\subsection{Points in $\pr^n$}
Let $I$ be the ideal of $s$ points in $\pr^n$.
Whereas $\alpha(I^{(m)})$ is very hard to compute for large $m$
except for very special configurations of the $s$ points,
if the points are general it is easy to show that
$$\alpha(I)=\min\left\{t:\, \binom{t+n}{n}-s>0\right\}.$$
The initial degree of the second symbolic power can also be computed,
by a result of Alexander and Hirschowitz \cite{AleHir95};
except for a few exceptions, they prove for $s$ general points in $\pr^n$ that
$$\alpha(I^{(2)})=\min\left\{t:\, \binom{t+n}{n}-s(n+1)>0\right\}.$$
For higher symbolic powers of ideals of $s\leq n+2$ general points,
$\alpha(I^{(m)})$ is in principle known for all $m$, although perhaps not explicitly
(see \cite[Proposition 2.3B]{refI}).
Also, for certain special configurations of points $\alpha(I^{(m)})$ is known for all $m$
and hence $\gamma(I)$ is known (see, for example, \cite[Example 8.3.4, Lemma 8.4.7]{refPSC}
for some examples in $\pr^n$, and see
\url{http://www.math.unl.edu/~bharbourne1/GammaFile.html}
for values of $\gamma(I)$ for all configurations of $s\leq 8$ points
in $\pr^2$, based on results of \cite{refGHM}).
However, little is typically known when $m$ is large for $s\geq n+3$ general points.
It is therefore perhaps surprising that one can establish exact values
of the Waldschmidt constant $\gamma(I)$
for all $s\leq n+3$ general points in $\pr^n$ for all $n$.

\begin{proposition}\label{pro:jn0s}
Let $J_{n,0,s}$ be the ideal of $s\leq n+3$ points in general position in $\PP^n$. Then
\begin{equation}
\label{exmd}
\gamma(J_{n,0,s}) = \begin{cases}
1, & s \leq n; \\
1+\frac{1}{n}, & s=n+1; \\
1+\frac{2}{n}, & s=n+2; \\
1+\frac{2}{n}, & s=n+3, n \text{ is even};\\
1+\frac{2}{n}+\frac{2}{n^3+2n^2-n}, & s=n+3, n \text{ is odd}.
\end{cases}
\end{equation}
\end{proposition}

Before we proceed with the proof, we need to introduce some notation.
Given $s$ general points $p_1,\dots,p_s$ in $\P^n$,
let $J(m_1,\dots,m_s)=I(p_1)^{m_1}\cap\dots\cap I(p_s)^{m_s}$ and let ${\mathcal J}(m_1,\dots,m_s)$
denote the corresponding sheaf of ideals on $\pr^n$. We use the convention
that if $I$ is an ideal in a ring $R$, then $I^{m}=R$ for $m\leq 0$.
We denote by
$$\sys_n(d;m_1,\dots,m_s)=H^0(\P^n,\calo_{\P^n}(d)\otimes {\mathcal J}(m_1,\dots,m_s))$$
and write $\ell_n(d;m_1,\dots,m_s)$ for the dimension of this linear system.
If some of the multiplicities are the same we abbreviate
$$\sys_n(d;m_1^{\times k_1},\dots,m_p^{\times k_p})=
\sys_n(d;\underbrace{m_1,\dots,m_1}_\text{$k_1$ times},\dots,\underbrace{m_p,\dots,m_p}_\text{$k_p$ times}).$$
We omit the index $n$ if it is clear from context.

\begin{lemma}\label{lem:cremona}
For $n\geq 2$, the Cremona transformation $(x_0:\ldots:x_n) \mapsto (\frac{1}{x_0}:\ldots:\frac{1}{x_n})$ of $\P^n$
induces a linear isomorphism
$$\sys_n(d;m_1,\dots,m_s)\mapsto \sys_n(d+c;m_1+c,\dots,m_{n+1}+c,m_{n+2},\dots,m_s),$$
where
$c=(n-1)d-\sum_{j=1}^{n+1} m_j$.
\end{lemma}
\proof See \cite{refMD} or \cite[(1.1)]{LafUga07}.
\endproof

We are now in a position to prove Proposition \ref{pro:jn0s}.

\proofof{Proposition \ref{pro:jn0s}}
If $n=1$, it is easy to see that $\gamma(J_{n,0,s})=s$ and the results hold, so assume $n\geq2$.

\textbf{Case $s\leq n$.} We have $\gamma(I)\geq 1$ for any nonzero homogeneous
ideal $I\subset k[\P^n]$. The claim then follows immediately as $\ell(d;d,\dots,d)\geq 1$.

\textbf{Case $s=n+1$.} By \cite[Lemma 2.4.1]{refBH},
$\alpha(J_{n,0,n+1}^{mn})=m(n+1)$, hence $\gamma(J_{n,0,s})=\lim_mm(n+1)/(mn)=1+(1/n)$.
We can also see this using Cremona transformations.
We have $\ell(1;1^{\times n},0)=1$, which implies
$\ell(n+1;n^{\times(n+1)})\geq 1$ (since $F=F_1\cdots F_{n+1}\in \sys_n(n+1;n^{\times(n+1)})$,
where $F_i$ is a linear form vanishing at the points $\left\{p_1,\dots,p_s\right\}\setminus\left\{p_i\right\}$),
and hence $\ell(h(n+1);(hn)^{\times(n+1)})\geq 1$
(since $F^h\in \sys_n(h(n+1);(hn)^{\times(n+1)})$). This implies $\alpha(J_{n,0,n+1}^{(hn)})\leq h(n+1)$.
We now show $\ell(h(n+1)-1;(hn)^{\times(n+1)})=0$ (and hence $\alpha(J_{n,0,n+1}^{(hn)})= h(n+1)$).
Indeed, Lemma \ref{lem:cremona} gives the first equality in
$$\ell(h(n+1)-1;(hn)^{\times(n+1)})=\ell(-n;(1-n-h)^{\times(n+1)})=0$$
and the second equality is obvious. This implies
$\alpha(J_{n,0,n+1}^{(hn)})/(hn)=h(n+1)/(hn)$ and since the Waldschmidt constant is the limit as $h\to\infty$,
we obtain $\gamma(J_{n,0,n+1})=1+\frac1n$.

\textbf{Case $s=n+2$.}
This can be done using Cremona transformations in a way
similar to the preceding case. Taking the product
$F=F_1\cdots F_{n+2}$ of linear forms $F_i$ vanishing along
$\left\{p_1,\dots,p_s\right\}\setminus\left\{p_i,p_{i+1}\right\}$
(where we read the indices of the points modulo $n+2$, so by
$p_{n+3}$ we mean $p_1$) we obtain a nonzero element $F^h\in\sys(h(n+2);(hn)^{\times(n+2)})$,
hence $\alpha(J_{n,0,n+2}^{(hn)})\leq h(n+2)$. We now apply Lemma \ref{lem:cremona} to show
$$\ell(h(n+2)-1;(hn)^{\times(n+2)})=\ell((h-1)n;((h-1)n-2h+1)^{\times(n+1)},hn).$$
The number on the right is $0$, since no nonzero polynomial
of degree $(h-1)n$ can have a zero of order $hn$.
Hence $\alpha(J_{n,0,n+2}^{(hn)})= h(n+2)$ and thus
$\gamma(J_{n,0,n+2})=\lim_{h\to\infty}h(n+2)/(hn)=1+\frac2n$.

\textbf{Case $s=n+3$.}

\textbf{Subcase $n=2m$.} The lower bound $\gamma(J_{n,0,n+3})\geq (n+2)/n=1+\frac2n$
follows immediately from the previous case. For the reverse inequality we
consider a sequence
$$\sys(q) = \sys\left(m+1-q;\, (m-q+1)^{\times 2q},\, (m-q)^{\times 2m+3-2q}\right),$$
$0\leq q\leq m$, of linear systems, so $\sys(0) = \sys\left(m+1; m^{\times 2m+3}\right)$
and  $\sys(m) = \sys\left(1; 1^{\times 2m},0^{\times 3}\right)$. Taking $\ell(q)=\dim \sys(q)$,
we obviously have $\ell(m)>0$, but applying Lemma \ref{lem:cremona}
to $\sys(q)$ using the $n+1$ points with the highest available multiplicities
gives $\ell(q)=\ell(q+1)$. (We note for each $q$ that $c=-1$, so applying Lemma \ref{lem:cremona}
converts $\sys(q)$ to $\sys(q+1)$.)
Thus $\ell(0)=\cdots=\ell(m)>0$, but $0<\ell(0)$ gives $\alpha(J_{n,0,n+3}^{(m)})\leq m+1$, so
$\alpha(J_{n,0,n+3}^{(hm)})\leq h(m+1)$ and therefore $\gamma(J_{n,0,n+3})\leq 1+\frac1m=1+\frac2n$.

\textbf{Subcase $n=2m+1$.} The proof is similar. For the upper bound on $\gamma(J_{n,0,n+3})$, consider
the linear systems
$$\sys(q) = \sys\left((m+1)(n+3)+qc;(m(n+3)+1+qc)^{\times n+3-2q}\right)$$
for $0\leq q\leq m+1$ where $c=-(n+1)$, so $\sys(0)=\sys\left((m+1)(n+3); (m(n+3)+1)^{\times n+3}\right)$ and
$\sys(m+1) = \sys\left(n+1; n^{\times n+1},(-1)^{\times 2}\right)$. Clearly $\ell(m+1)>0$ (just consider the
$n+1$ coordinate hyperplanes taken together),
and again, applying Lemma \ref{lem:cremona} using the $n+1$
points of highest available multiplicity converts $\sys(q)$ into $\sys(q+1)$,
so we find $\ell(q)=\ell(q+1)$. Thus $\ell(0)>0$ and hence
$$\alpha(J_{n,0,n+3}^{(m(n+3)+1)})\leq (m+1)(n+3),$$
which gives
$$\gamma(J_{n,0,n+3})\leq \frac{(m+1)(n+3)}{m(n+3)+1}=
\frac{(n+1)(n+3)}{(n-1)(n+3)+2}=1+\frac2n+\frac{2}{n^3+2n^2-n}.$$
For the reverse inequality it suffices to check that
$$\ell\left(h(2m+2)(2m+4)-1;\, (h(2m(2m+4)+2))^{\times 2m+4}\right)=0$$
for $h\geq 1$. This time we consider
$$\sys(q) = \sys(d(q);m_1(q)^{\times (2m+4-2q)},m_2(q)^{\times (2q)}),$$
for $0\leq q\leq m+2$, where
$d(q) = ((2m+2)(2m+4)-1+qc$, $m_1(q) = (2m(2m+4)+2+qc$ and
$m_2(q) = (2m(2m+4)+2+(q-1)c$, where $c=-(4m+4)h-2m$.
Applying Lemma \ref{lem:cremona} using the $n+1$
points of highest available multiplicity converts $\sys(q)$ into $\sys(q+1)$,
so $\ell(q)=\ell(q+1)$.
What we want is to show $\ell(0)=0$, and we know $\ell(0)=\cdots=\ell(m+2)$,
but $\sys(m+2)$ has degree $d(m+2)=h(2m+2)(2m+4)-1+(m+2)c=-2m(m+2)-1<0$,
so $\ell(m+2)=0$, as desired.
\endproofof

\begin{example}\label{infptsinP3}
Here we verify that $e_{3,0,4}=3/2$.
Note that $P_{3,0,4,m}(t)=\binom{t+3}{3}-4\binom{m+2}{3}$, and
so $P_{3,0,4,2}(3)=16>0$. Thus the $\inf$ is at most $3/2$.
But $6P_{3,0,4,m}(mx)=m^3(x^3-4)+m^2(6x^2-12)+m(11x-8)$. We must verify
that there is no value of $x$ in the range $1\leq x<3/2$ such that
$x=t/m$ for integers $t\geq m\geq 1$ with
$P_{3,0,4,m}(t)>0$. But $x^3-4$, $6x^2-12$ and $11x-8$ are all
increasing in the range $1\leq x\leq 3/2$, while $6P_{3,0,4,m}(m3/2)$
is negative for $m\geq 6$. Thus we need check only values $t,m$
such that $m\leq t\leq 3m/2$ for $1\leq m\leq 6$.
Doing so shows that the infimum is indeed $3/2$.
\end{example}

\setcounter{theorem}{0}

\subsection{Lines in $\pr^n$, $n\geq 3$}
The initial degree of the ideal $J_{n,1,s}$ of a general union of $s$ lines in $\P^n$ for $n\geq3$
can be computed explicitly by a result of Hartshorne and Hirschowitz
\cite[Theoreme 0.1]{HarHir81}, which gives
$$\alpha(J_{n,1,s})=\min\left\{t:\, \binom{n+t}{n}-s(t+1)>0\right\}.$$
As in the case of points, computing $\alpha(J_{n,1,s}^{(m)})$ for $m>1$
is in general an open problem, but unlike points it is open even for $m=2$. Certainly we have
$$\alpha(J_{n,1,s}^{(2)})\leq\min\left\{t:\, \binom{n+t}{n}-sc_{n,1,2,t}>0, t\geq 2\right\}$$
but it is not known when equality fails or by how much (although it is known that
it can fail; for example, $\alpha(J_{3,1,3})=2$ implies
$\alpha((J_{3,1,3})^{(2)})\leq4$ but $\min\left\{t:\, \binom{3+t}{3}-sc_{3,1,2,t}>0\right\}=5$).

Now we consider various specific cases for general lines in $\PP^3$.
In the next result, note in each case that $\gamma_{3,1,s}\leq g_{3,1,s}$,
as asserted by Theorem \ref{thm:zeroes of Lambda}.

\begin{proposition}\label{5genericlines}
The following holds:
$$\begin{array}{c|c|c}
s & \gamma_{3,1,s} & g_{3,1,s} \\ \hline
1 & 1 & 1 \\
2 & 2 & 2 \\
3 & 2 & \approx 2.584 \\
4 & 8/3 & \approx 3.064 \\
5 & 10/3 & \approx 3.482
\end{array}$$
\end{proposition}

\begin{proof}
The case $s=1$ is easy. For $s=2,3$ and 4, see \cite[Corollary 1.1]{refGHVT2} (version 2).
For $s=5$, let $I=J_{3,1,5}$ be the ideal of $s=5$ general lines $L_1,\ldots,L_5\subset\pr^3$.
We can easily see that $\gamma(I)\leq 10/3$, since there is a unique
quadric $Q_{ijl}$ containing lines $L_i,L_j,L_l$ for each choice of
$1\leq i<j<l\leq 5$. Thus $Q_{123},Q_{234},Q_{345},Q_{145},Q_{125}$
together give a hypersurface $H$ of degree 10 with multiplicity $m=3$ along each
of the five lines. Thus $\gamma(I)\leq \deg(H)/m=10/3$. To prove that
$\gamma(I)=10/3$ it will suffice to show that $\alpha(I^{(3m)})\geq 10m$;
i.e., that $(I^{(3m)})_{10m-1}=0$.

To do this blow up the five lines to obtain the morphism $\pi:X\to\pr^3$.
Let $E_i=\pi^{-1}(L_i)$. Let $D_{ijl}$ be the proper transform of $Q_{ijl}$.
We may assume the lines $L_i$, $L_j$ and $L_l$ are $(1,0)$-classes on $Q_{ijl}$.
Note that $D_{123}$ is isomorphic to the blow up of $Q_{123}$
at the 2 points of intersection of $Q_{123}$ with $L_4$
and also at the 2 points of intersection with $L_5$.
We will denote by $C_{ijl}(a,b)$ a divisor of type $(a,b)$ on $Q_{ijl}$, and we will abuse notation
and use that also for its pullback to $D_{ijl}$.
Thus the divisor class group
of $D_{ijl}$ is spanned by the classes of $C_{ijl}(0,1)$, $C_{ijl}(1,0)$,
and by the classes of the exceptional divisors $e_1,e_2,e_3$ and $e_4$ of the
four points of $Q_{ijl}$ blown up. (Since it will always be clear which $D_{ijl}$ we're referring to,
we will not need to use notation to distinguish $e_1$ on $D_{ijl}$
from $e_1$ on another divisor $D_{i'j'l'}$.)
We will use $H$ to denote the hyperplane class on $\pr^3$ and also its pullback to
$X$. We will use $E$ to denote $E_1+\cdots+E_5$ and $\sim$ to denote linear equivalence.
Note that $D_{ijl}\sim 2H-E_i-E_j-E_l$ and so $D_{123}+D_{145}+D_{234}+D_{125}+D_{345}\sim 10H-3E$.

We have the following short exact sequences, where $F_t=(10t-1)H-3tE$:

{\small

\[
\begin{split}
0&\to{\mathcal O}_X(F_t-D_{123})\to{\mathcal O}_X(F_t)\to {\mathcal O}_{D_{123}}((t-1,10t-1)-3t(e_1+\cdots+e_4))\to0;\\
\\
0&\to {\mathcal O}_X(F_t-D_{123}-D_{145})\to{\mathcal O}_X(F_t-D_{123})\\
&\to{\mathcal O}_{D_{145}}((t-2,10t-3)-(3t-1)(e_1+\cdots+e_4))\to0;\\
\\
0&\to {\mathcal O}_X(F_t-D_{123}-D_{145}-D_{234})\to {\mathcal O}_X(F_t-D_{123}-D_{145}) \\
&\to{\mathcal O}_{D_{234}}((t-2,10t-5)-(3t-2)(e_1+e_2)-(3t-1)(e_3+e_4))\to0;\\
\\
0 & \to
{\mathcal O}_X(F_t-D_{123}-D_{145}-D_{234}-D_{125}) \to {\mathcal O}_X(F_t-D_{123}-D_{145}-D_{234}) \\
&\to {\mathcal O}_{D_{125}}((t-2,10t-7)-(3t-2)(e_1+\cdots+e_4))\to0;\\
\\
0&\to {\mathcal O}_X(F_{t-1}) \to {\mathcal O}_X(F_t-D_{123}-D_{145}-D_{234}-D_{125}) \\
&\to{\mathcal O}_{D_{345}}((t-3,10t-9)-(3t-3)(e_1+\cdots+e_4))\to0.
\end{split}
\]

}
It is easy to check that $|C_{ijl}(1,2)-e_1-\cdots-e_4|$ is nonempty and fixed component free,
and hence it is nef on $D_{ijl}$ but the intersection with the classes
on the right end of each short exact sequence is $-3$ in each case
(for example, for the third exact sequence, we have
$(C_{234}(t-2,10t-5)-(3t-2)(e_1+e_2)-(3t-1)(e_3+e_4))\cdot (C_{234}(1,2)-e_1-\cdots-e_4)=-3$)), so
we see that $h^0=0$ for the ${\mathcal O}_{D_{ijl}}$ term in each exact sequence.
Thus $h^0({\mathcal O}_X(F_{t-1}))=0$ implies
all of the $h^0$ terms in all of the sequences vanish.
But $h^0({\mathcal O}_X(F_{0}))=h^0({\mathcal O}_{\pr^3}(-1))=0$,
so by induction we see $h^0({\mathcal O}_X(F_{t}))=0$ for all $t\geq 0$.
\end{proof}

\begin{remark}\label{6genericlines}
Consider the ideal $I$ of 6 generic lines in $\pr^3$. Then
$P_{3,1,6,7}(27)=28$ is positive,
so $\gamma(I)\leq t/m=27/7 \approx 3.857142<g_{3,1,6} \approx 3.85878$,
but we can get an even better bound, coming from six forms of degree 7,
each vanishing to order 1 on different choices of one of the lines
and order 2 on the other five giving $\gamma(I)\leq 42/11 \approx 3.818$.
Computer calculations suggest that there is a form of degree 12
vanishing with order 4 on one line and order 3 on the five other general lines,
although $P_{3,1,(4,3,3,3,3,3)}(12)<0$, so we would ``expect'' that no such form should exist.
Given that it does exist, averaging the orders of vanishing gives an even better
bound of $\gamma(I)\leq 72/19 \approx 3.78947$.
\end{remark}

Proposition \ref{5genericlines} and Remark \ref{6genericlines}
give examples for which $\gamma_{3,1,s} < g_{3,1,s}$.
What underlies these examples
is the occurrence of integers $m_1,\ldots,m_s\geq 0$ and $d\geq \max_i(m_i)>0$
with $P_{3,1,(m_1,\ldots,m_s)}(d)>0$ but such that
$d/m<g_{3,1,s}$, where $m$ is the average of the $m_i$.
Having $P_{3,1,(m_1,\ldots,m_s)}(d)>0$ guarantees by Lemma \ref{lem:HilbPolyn}
the occurrence of a nonzero
form $F$ of degree $d$ vanishing to order $m_i$ on line $L_i$
for generic lines $L_1,\ldots,L_s\subset\pr^3$.
This implies $\gamma_{3,1,s}\leq d/m$, so we obtain
$\gamma_{3,1,s} < g_{3,1,s}$. For example, consider the result of
Proposition \ref{5genericlines} for $s=5$ general lines. In this case
$P_{3,1,(1,1,1,0,0)}(2)=4>0$, so there is a quadric,
$Q$, containing three of the five lines, so $d=2$ and $m=3/5$. This gives
$\gamma_{3,1,s} \leq d/m=2/(3/5)=10/3< 3.582\approx g_{3,1,s}$.
Finding such examples for infinitely many $s$ would contradict Conjecture {\conjC}.
However, the following result shows that no such reducible examples, built out of
components corresponding to positive values of a Hilbert polynomial, occur
when $s\geq 7$. (This does not rule out the possibility of reducible submaximal examples
whose components correspond to negative values of a Hilbert polynomial,
such as what seems to be the case for the upper bound of
$72/19$ in Remark \ref{6genericlines} for six general lines.)

\begin{theorem}
\label{nosymetry}
Let ${\bf v}=(m_1,\ldots,m_s)$ be an integer vector with $s\geq7$ entries $m_j\geq 0$, but not all 0,
such that $d\geq m_j$ for all $j$. Let $m=(\sum_j m_j)/s$. If $P_{3,1,{\bf v}}(d)> 0$, then
$$\frac{d}{m} \geq g_{3,1,s}.$$
\end{theorem}

\begin{proof}
Since $d\geq m_j$, we have $P_{3,1,{\bf v}}(d)=\binom{d+3}{3}-\sum_jc_{3,1,m_j,d}$,
and by Lemma \ref{combinatorialLemma2}(c), we have
$c_{3,1,m_j,d}=\frac{1}{6}(3d-2m_j+5)m_j(m_j+1)$. To show
the contrapositive, given $d/m< g_{3,1,s}$, it suffices to show that
\begin{equation}\label{bettereq}
\binom{d+3}{3} - \sum_j \frac{1}{6}(3d-2m_j+5)m_j(m_j+1) < 1.
\end{equation}
It is easy to check that
$(3d-2m+5)m(m+1)/6$ is an increasing function of $m$ on the interval $(0,d)$.
Thus the terms in the sum $\sum_j \frac{1}{6}(3d-2m_j+5)m_j(m_j+1)$
are nonnegative.

First consider the case when for some $j_1\neq j_2$ we have $m_{j_1}+m_{j_2} > d$.
Without loss of generality we may assume that $j_1=1$, $j_2=2$ and $m_1\leq m_2\leq d$.
Hence $P_{3,1,{\bf v}}(d)$ is bounded above by
$\binom{d+3}{3} - c_{3,1,m_1,d}-c_{3,1,m_2,d}$.
If we write $t=m_1+m_2-d-1$, then $d=m_1+m_2-t-1$ and $0\leq t\leq m_1-1$.
By substituting $d=m_1+m_2-t-1$ into $\binom{d+3}{3} - c_{3,1,m_1,d}-c_{3,1,m_2,d}$
we obtain
\[
\begin{split}
P_{3,1,{\bf v}}(d)\leq &\binom{d+3}{3} - c_{3,1,m_1,d}-c_{3,1,m_2,d}\\
=&-6m_2m_1t+3m_2t^2+3m_1t^2-3m_2t-3m_1t-t^3+3t^2-2t\\
=&-3t(2m_1m_2-(m_1+m_2)t)-3m_2t-3m_1t-t(t-1)(t-2)\leq0.
\end{split}
\]
To justify the inequality, we just must verify that each term is nonpositive.
But $t(t-1)(t-2)$ is nonnegative for integer values of $t\geq 0$, and
$2m_1m_2-(m_1+m_2)t\geq 2m_1m_2-(m_1+m_2)(m_1-1)=(m_2-m_1)m_1+m_1+m_2\geq0$,
since $0\leq t\leq m_1-1$. Thus $P_{3,1,{\bf v}}(d)\leq0$ in this case.

For the next case, assume that $m_{j_1}+m_{j_2} \leq d$ for all $j_1\neq j_2$.
We now will treat $P_{3,1,{\bf v}}$ as a polynomial in the $m_j$, so values for each
$m_j$ need not be integers.
Observe that $P_{3,1,{\bf v}}(d)$ has the greatest possible value if
the sum in \eqref{bettereq}
has the smallest possible value. Consider the situation when not all $m_j$'s are equal. Without
loss of generality, assume $0\leq m_1= m_2= \cdots=m_i <m_{i+1}\leq m_{i+2}\leq\cdots\leq m_s$.
Regard $\frac{i}{6}(3d-2(m_1+\varepsilon)+5)(m_1+\varepsilon)(m_1+\varepsilon+1)
+ \frac{1}{6}(3d-2(m_{i+1}-i\varepsilon)+5))(m_{i+1}-i\varepsilon)(m_{i+1}-i\varepsilon+1)$
as a polynomial in $\varepsilon\geq0$. Its derivative
with respect to $\varepsilon$ evaluated at $\varepsilon=0$ is equal to
$$6i(m_1+m_{i+1}-d-1)(m_{i+1}-m_1),$$
which is negative.
It follows that replacing $m_1,\ldots,m_i$ by $m_1+\varepsilon,\ldots,m_i+\varepsilon$
and $m_{i+1}$ by $m_{i+1}-i\varepsilon$
for small positive values of $\varepsilon$ does not change $m_1+\cdots+m_{i+1}$
(and hence leaves $m$ unchanged) but increases
$P_{3,1,{\bf v}}(d)$. Thus it is enough to consider the case
that $m_1=\cdots=m_{i+1}$, and repeating the argument we reduce to the case that
$m_1=\cdots=m_s=m$.

By the substitution $d=m\tau$, \eqref{bettereq} is equivalent to
$$m^3(\tau^3-3s\tau+2s)+m^2(6\tau^2-3s\tau-3s)+m(11\tau-5s) < 0.$$
Assume $1\leq\tau\leq g_{3,1,s}$. Then $\tau^3-3s\tau+2s\leq 0$, so it is enough to show that
$$(6\tau^2-3s\tau-3s)m+(11\tau-5s) < 0.$$
But $(6\tau^2-3s\tau-3s)m+(11\tau-5s)$ is quadratic in $\tau$, so to show it is negative
it is enough to check negativity for $\tau=1$ and $\tau=g_{3,1,s}$.
For $\tau=1$ it evaluates to $(6-6s)m+(11-5s)$, which is negative for $s\geq 3$
(and so certainly for $s\geq 7$).
Thus now it is enough to show $(6g^2-3sg-3s)m+(11g-5s) < 0$
for $g=g_{3,1,s}$.

To do so, note that $6g^2-3sg-3s<0$ for $s\geq 11$. (This is because $\Lambda_{3,1,s}(s/2)>0$
for $s\geq 11$, so $g\leq s/2$ for $s\geq 11$ by Lemma \ref{lem:zeroes of Lambda}(b),
but $g\leq s/2$ implies $6g^2-3sg-3s<0$.) Similarly, $11g-5s<0$ for $s\geq 13$ (since
$\Lambda_{3,1,s}(5s/11)>0$ for $s\geq 13$). Thus the result is proved for $s\geq 13$.

For $s=7,\ldots,12$, we can check by explicit computation, that
$6g^2-3sg-3s<0$. Thus $(6g^2-3sg-3s)m+(11g-5s) < 0$ is equivalent to
$$m > - \frac{11g-5s}{6g^2-3sg-3s}.$$
Therefore the cases that are left have $7\leq s\leq 12$ and
$m \leq - \frac{11g-5s}{6g^2-3sg-3s}$.
Hence for $7\leq s\leq 12$, we need only check \eqref{bettereq} for
cases such that $d<-g\frac{11g-5s}{6g^2-3sg-3s}$ and
$ds/g<\sum_jm_j\leq -s\frac{11g-5s}{6g^2-3sg-3s}$,
where now the $m_j$ and $d$ are integers
with $m_{j_1}+m_{j_2}\leq d$ for all $j_1\neq j_2$.
By direct calculation we find the following data.

$$\begin{array}{c|c|c|c|c}
s & g & \text{upper bound on }d & \text{upper bound on }\sum_jm_j \\ \hline
7 & 4.2035 & 14.5043 & 24.1538 \\
8 & 4.5236 & 4.51017 & 7.97625 \\
9 & 4.82374 & 2.20558 & 4.11512 \\
10 & 5.10725 & 1.18148 & 2.31334 \\
11 & 5.37664 & 0.602377 & 1.23239 \\
12 & 5.63383 & 0.229665 & 0.489184
\end{array}$$

If $d=1$, then $d\geq m_j$ implies $m_j\leq 1$ for all $j$,
but $m_{j_1}+m_{j_2}\leq d$ now implies (up to reindexing) that
$m_1=1$ and all other $m_j=0$, in which
case $d/m<g$ implies $s<g$, which we see from the table does not happen.
Thus we may assume $d\geq 2$. In particular, cases $s=10,11,12$ are immediate.
For $s=9$, we must have $d=2$ and $3.73=ds/g\leq\sum_jm_j\leq 4.1$,
so $\sum_jm_j=4$. There are three cases:
$m_1=\cdots=m_4=1$ and all other $m_j=0$;
$m_1=2, m_2=m_3=1$ and all other $m_j=0$; and
$m_1=m_2=2$ and all other $m_j=0$.
In each case we find $P_{3,1,{\bf v}}(d)\leq0$.
For $s=8$, either $d=2$ or $d=3$; for $d=2$ there are 14 possibilities for the $m_j$
and for $d=3$ there are 15 possibilities. In each case $P_{3,1,{\bf v}}(d)\leq0$.

Finally, for $s=7$ there are 4149 possible multiplicity sequences
$0\leq m_1\leq \cdots\leq m_7$ with $7m=m_1+\cdots+m_7\leq 24$.
For each such multiplicity sequence, and for each $2\leq d\leq 14$ such that
$m_7\leq d$ and $d/m<g$, it is easy to write a computer program to check
that $P_{3,1,{\bf v}}(d)\leq0$, which verifies the case $s=7$.
\end{proof}

\begin{example}\label{inflinesinP3}
Here we verify that $\inf\left\{\frac{t}{m}:\, t\geq m\geq 1,\, P_{3,1,6,m}(t)>0 \right\}=27/7$.
Note that $P_{3,1,6,m}(t)=\binom{t+3}{3}-(3t-2m+5)(m+1)m$, and
so $P_{3,1,6,7}(27)>0$. Thus the $\inf$ is at most $27/7$.
But $6P_{3,1,6,m}(mx)=m^3(x^3-18x+12)+m^2(6x^2-18x-18)+m(11x-30)$.
Note that all terms are negative when $x<30/11$. Thus it is enough to verify
that there is no value of $x$ in the range $30/11\leq x<27/7$ such that
$x=t/m$ for integers $t\geq m\geq 1$ with
$P_{3,1,6,m}(t)>0$. But $x^3-18x+12$, $6x^2-18x-18$ and $11x-30$ are all
increasing in the range $30/11\leq x\leq 27/7$, while $6P_{3,1,6,m}(m27/7)$
is negative for $m\geq 50$. Thus we need check only values $t,m$
such that $30m/11\leq t\leq 27m/7$ for $1\leq m\leq 50$.
Doing so shows that the $\inf$ is indeed $27/7$.
\end{example}

\setcounter{theorem}{0}

\subsection{Additional Examples}

Let $I$ be the ideal of $s$ generic $r$-planes in $\pr^n$ with $n\geq 2r+1$.
When $s=1$ we have $\gamma(I)=1$ (since $I^{(m)}=I^m$, hence $\gamma(I)=\alpha(I)$),
and we also have $g_{n,r,1}=1$.

As another example, using \eqref{eq:LambdaforLines}, when $n=3$, we have
$g_{3,1,2}=2$, and indeed for the ideal $I$ of two skew lines $L_1,L_2\subset\pr^3$ we have
$\gamma(I)=2$. (This is because $I^{(m)}=I(L_1)^m\cap I(L_2)^m=(I(L_1)I(L_2))^m$
by \cite{refGHVT}, hence $\alpha(I^{(m)})=m(\alpha(I(L_1))+\alpha(I(L_2)))=2m$,
so $\gamma(I)=2$.) This recovers the $s=2$ case of Proposition \ref{5genericlines}.
This is a special case of $s=(n-1)^{n-2}$ lines in $\pr^n$ for which we showed in section
\ref{section4} that $\gamma(I)=g_{n,1,s}=n-1$. It is also a special case of another infinite family.
Let $I$ be the ideal of two generic $r$-planes in $\pr^n$ for $n=2r+1$. We then
have $\gamma(I)=2$ (since $I^{(m)}=I^m$ \cite[Lemma 4.1]{refGHVT}, hence $\gamma(I)=\alpha(I)$).
We now show that $g_{2r+1,r,2}=2$ for all $r\geq 1$,
extending the $s=2$ case of Proposition \ref{5genericlines} to arbitrary $r\geq 1$.

\begin{lemma}
For all $r\geq 1$, we have $g_{2r+1,r,2}=2$.
\end{lemma}

\begin{proof}
It is enough by Theorem \ref{thm:zeroes of Lambda} to show that
$\Lambda_{2r+1,r,2}(2)=0$. By Proposition \ref{lem:Lambdan0s}, we have
$$\Lambda_{2r+1,r,2}(2)=\frac{1}{(2r+1)!}\Big(2^{2r+1}-2\sum_{j=0}^r\binom{2r+1}{j}\Big),$$
but
$2\sum_{j=0}^r\binom{2r+1}{j}$ can be written as
$\Big(\binom{2r+1}{0}+\cdots+\binom{2r+1}{r}\Big)
+\Big(\binom{2r+1}{r}+\cdots+\binom{2r+1}{0}\Big)$
and thence as
$\Big(\binom{2r+1}{0}+\cdots+\binom{2r+1}{r}\Big)
+\Big(\binom{2r+1}{r+1}+\cdots+\binom{2r+1}{2r+1}\Big)
=(1+1)^{2r+1}$, and so $\Lambda_{2r+1,r,2}(2)=0$.
\end{proof}

\paragraph*{\emph{Acknowledgement.}}
   We would like to thank Armen Edigarian, Alex Kuronya, Thomas Bauer, Joaquim Ro\'e and Ciro Ciliberto for helpful discussions.
   The second author's work on this project
was sponsored by the National Security Agency under Grant/Cooperative
agreement ``Advances on Fat Points and Symbolic Powers,'' Number H98230-11-1-0139.
The United States Government is authorized to reproduce and distribute reprints
notwithstanding any copyright notice.
   Parts of this paper
   were written while the third author was a visiting professor at the University
   Mainz as a member of the program Schwerpunkt Polen. Szemberg's research
   was partially supported by NCN grant UMO-2011/01/B/ST1/04875.


\begin{thebibliography}{99}

\bibitem{AleHir95}
J.\ Alexander and A.\ Hirschowitz.
\emph{Polynomial interpolation in several variables},
J. Algebraic Geom. 4 (1995), 201--222

\bibitem{refPSC}
T.\ Bauer, S.\ Di Rocco, B.\ Harbourne, M.\ Kapustka, A.\ Knutsen, W.\
Syzdek, and T.\ Szemberg.
\emph{A primer on Seshadri constants},
Interactions of Classical and Numerical Algebraic Geometry,
Proceedings of a conference in honor of A.\ J.\ Sommese, held at Notre Dame, May 22--24 2008.
Contemporary Mathematics vol.\ 496, 2009, eds.
D.\ J.\ Bates, G-M.\ Besana, S.\ Di Rocco, and C.\ W.\ Wampler,  362 pp.
(arXiv:0810.0728).

\bibitem{Boc05}
   Bocci,\ C.
   \emph{Special effect varieties in higher dimension},
   Collect. Math. 56 (2005), 299--326.

\bibitem{refBH} C.\ Bocci and B.\ Harbourne.
\emph{Comparing Powers and Symbolic Powers of
Ideals}, J. Algebraic Geometry, 19 (2010) 399--417.

\bibitem{refDer}
H.\ Derksen.
\emph{Hilbert series of subspace arrangements},
J.\ Pure Appl.\ Algebra, 209 (2007) 91--98.

\bibitem{refMD} M.\ Dumnicki.
\emph{An algorithm to bound the regularity and nonemptiness of linear systems in $\PP^n$},
J. Symbolic Comput., 44 (2009), no. 10, 1448--1462.

\bibitem{EV83}
H.\ Esnault and E.\ Viehweg.
\emph{Sur une minoration du degr\'e d'hypersurfaces s'annulant en certains points},
Math. Ann. 263 (1983), 75--86.

\bibitem{Evain}
L.\ Evain.
\emph{On the postulation of $s^d$ fat points in $\pr^d$},
J. Algebra {\bf 285} (2005), 516--530.

\bibitem{refGHM} A.\ Geramita, B.\ Harbourne and J.\ Migliore.
\emph{Classifying Hilbert functions of fat point subschemes in $\pr^2$},
Collect. Math. 60, 2 (2009), 159--192.

\bibitem{refGHM2} A.\ Geramita, B.\ Harbourne and J.\ Migliore.
\emph{Star configurations in $\P^n$},
preprint (2012), arXiv:1203.5685.

\bibitem{refGHVT} E.\ Guardo, B.\ Harbourne and A.\ Van Tuyl.
\emph{Asymptotic resurgences for ideals of positive dimensional
subschemes of projective space}, preprint, 11pp., arXiv:1202.4370.

\bibitem{refGHVT2} E.\ Guardo, B.\ Harbourne and A.\ Van Tuyl.
\emph{Symbolic powers versus regular powers of ideals
of general points in $\pr^1\times\pr^1$},
preprint, 20pp., arXiv:1107.4906 (versions 1 and 2),
to appear, Canad.\ J.\ Math.

\bibitem{Har92}
J.\ Harris,
Algebraic geometry. A first course.
Graduate Texts in Mathematics, 133. Springer-Verlag, New York, 1992

\bibitem{HarHir81}
R.\ Hartshorne and A.\ Hirschowitz.
\emph{Droites en position g\'en\'erale dans l'espace projectif},
Algebraic geometry (La R\'abida, 1981), 169--188, Lecture Notes in Math., 961, Springer, Berlin, 1982.

\bibitem{refI}
A.\ Iarrobino.
\emph{Inverse system of a symbolic power III: thin algebras and fat points},
Compositio Math. 108, (1997), 319--356.

\bibitem{LafUga07}
   A.\ Laface, L.\ Ugaglia:
\emph{Elementary $(-1)$--curves of $\P^3$},
   Comm. Algebra 35 (2007), 313--324

\bibitem{Magg00}
M.\ Maggesi.
\emph{On the quantum cohomology of blow-ups of projective spaces along linear subspaces},
arXiv:math/9810150v2, preprint 2000.


\bibitem{refN}
M.\ Nagata.
\emph{On the 14-th problem of Hilbert}. Amer. J. Math. 81 (1959), 766--772.

\bibitem{refW}
M.\ Waldschmidt.
\emph{Propri\'et\'es arithm\'etiques de fonctions de plusieurs variables II},
S\'eminaire P.\ Lelong (Analyse), 1975--76, Lecture Notes Math. 578,
Springer-Verlag, 1977, 108--135.

\end{thebibliography}
\end{document}